\documentclass[final,5p,times,twocolumn]{elsarticle}
\usepackage{epsfig}
\usepackage{amsmath,amssymb,amsfonts}
\usepackage{amsmath}
\usepackage{amsthm}
\usepackage{algorithm}
\usepackage{algpseudocode}
\usepackage{caption}
\usepackage{subcaption}
\DeclareCaptionLabelFormat{simple}{#1~#2}
\usepackage{graphicx}
\usepackage{geometry}
\usepackage{booktabs, tabularx, makecell, multirow}
\usepackage{balance}
\usepackage[capitalize]{cleveref}
\usepackage{microtype}

\newtheorem{remark}{Remark}
\newtheorem{theorem}{Theorem}
\newtheorem{assumption}{Assumption}

\emergencystretch=\maxdimen
\def\BibTeX{{\rm B\kern-.05em{\sc i\kern-.025em b}\kern-.08em
		T\kern-.1667em\lower.7ex\hbox{E}\kern-.125emX}}

\begin{document}

\begin{frontmatter}         

\title{A Taylor-Bernstein Inner Approximation Algorithm for Path-Constrained Dynamic Optimization\tnoteref{t1}}

\tnotetext[t1]{This work was supported by Scientific Research Fund of Educational Department of Liaoning Province of China (LJ212510167008), National Natural Science Foundation of China (61873041), and Natural Science Foundation of Liaoning Province of China (2024-MS-184).}

\author[inst1]{Yuan Chang}
\ead{2024008001@bhu.edu.cn}

\author[inst2]{Lizhong Jiang}
\ead{LizhongJiang_adm@163.com}

\author[inst1]{Tai-Fang Li\corref{cor1}}
\ead{taifang0416@bhu.edu.cn}

\author[inst2]{Jun Fu}
\ead{junfu@mail.neu.edu.cn}

\cortext[cor1]{Corresponding author}

\affiliation[inst1]{organization={College of Control Science and Engineering, Bohai University},
	city={Jinzhou},
	postcode={121013},
	state={Liaoning},
	country={China}}

\affiliation[inst2]{organization={State Key Laboratory of Synthetical Automation for Process Industries, Northeastern University},
	city={Shenyang},
	postcode={110819},
	state={Liaoning},
	country={China}}

\begin{abstract}
A novel inner approximation algorithm is proposed for dynamic optimization problems to ensure strict satisfaction of path constraints. Distinct from traditional methods relying on interval analysis, the proposed algorithm leverages the convex hull property of Bernstein polynomials to tightly bound the polynomial components of the Taylor expansion, while incorporating the Log-Sum-Exp technique to smooth the non-differentiability arising from coefficient maximization. This approach yields a tighter upper bound function compared to interval methods, with a smaller approximation error. Theoretical analysis shows that the algorithm converges in a finite number of steps to a KKT solution of the original problem that satisfies the specified tolerances. Numerical simulations confirm that the proposed algorithm effectively reduces the number of constraints in the approximation problem, improving computational performance while ensuring strict feasibility.

\end{abstract}


\end{frontmatter}

\section{Introduction}
Dynamic optimization is widely applied in chemical engineering \cite{1}, robotics engineering \cite{2}, and trajectory optimization \cite{3}, aiming to identify optimal control strategies for complex systems. In practical engineering scenarios, system states are typically constrained by physical limits or safety regulations \cite{30,31}. These constraints manifest as path constraints that must be strictly satisfied throughout the entire time horizon, while any violation may lead to equipment damage or even catastrophic consequences. However, path constraints are inherently infinite-dimensional constraints defined over continuous time. Efficiently transforming them into finite- dimensional forms while ensuring the strict feasibility of numerical solutions remains a major challenge in this field.

Currently, mainstream numerical methods for solving dynamic optimization problems primarily rely on direct transformation techniques \cite{4,5}. By discretizing the control input (or simultaneously discretizing the state trajectory) into finite-dimensional decision variables, the original optimal control problem is transformed into a Nonlinear Programming (NLP) problem and solved by mature NLP solvers. However, within this framework, infinite-dimensional path constraints defined in the continuous time domain are typically relaxed into algebraic inequalities at a finite number of discrete grid points. This approach to handling point constraints suffers from a fundamental flaw: it offers no theoretical guarantee that the constraints are strictly satisfied between discrete points. To address this issue, \cite{6} and \cite{7} proposed converting path constraints into integral forms to suppress violations by limiting the accumulation of constraint violations over time. However, since the gradient of the integral constraint vanishes within the feasible region, the constraint bound is typically relaxed to a small positive tolerance, leading to slight constraint violations.

To ensure the strict feasibility of path constraints, \cite{8} proposed a deterministic global optimization method based on a branch-and-bound framework, utilizing validated integration techniques to verify constraint satisfaction in dynamic systems within the continuous-time domain. However, this approach is limited by computational complexity and is typically applicable only to low-dimensional problems. In contrast, \cite{9} introduced a semi-infinite programming approach that iteratively restricts the right-hand side of path constraints and incorporates points with maximum violations into the constraint set, thereby progressively eliminating violations and achieving finite-step convergence. This strategy was subsequently successfully extended to switched systems \cite{10,11} and model predictive control \cite{12,13}. Building upon this foundation, \cite{14} enhanced computational efficiency by constructing polynomial approximations between adjacent nodes and imposing constraints at their extrema, albeit at the cost of increasing the number of constraints.

Unlike methods that iteratively restrict the right-hand side of constraints, \cite{15,16,17} have developed a distinct research path. The core philosophy of this approach lies in approximating the original problem strictly from the interior of its feasible region. This offers a significant advantage: every solution generated during the iterative process is guaranteed to satisfy the original path constraints. \cite{15} employed $\alpha BB$ techniques to adaptively convexify and relax the sub-optimization to construct an inner approximation algorithm; however, the number of decision variables and constraints scales linearly with the number of subintervals. \cite{16} proposed a “safe envelope” technique based on Bernstein polynomials, achieving strict constraint satisfaction under orthogonal configurations while avoiding excessive conservatism. However, this method is limited to polynomial-form constraints. Recently, \cite{17} combined Taylor models with interval analysis to construct upper bound functions for path constraints. By solving restricted problems and employing adaptive interval refinement strategies, finite convergence properties under strict feasibility were established. Moreover, this method handles general nonlinear constraints and achieves higher computational efficiency than \cite{15} by maintaining a smaller problem scale. However, this method inevitably suffers from dependency effects during interval operations \cite{18}, resulting in loosely constructed upper bound functions that often force algorithms to perform excessive grid refinement for convergence. Therefore, effectively suppressing dependency effects to construct tighter upper bound functions is crucial for enhancing the efficiency of these algorithms.

To effectively suppress the dependency effect that causes computational inefficiency, the Taylor-Bernstein form is a powerful solution. In the field of global optimization, this technique has been successfully validated \cite{19,20,21}, where the convex hull property of Bernstein polynomials is used to construct tight envelopes, thereby significantly reducing the conservatism introduced by interval arithmetic. However, despite its success in static optimization, the application of the Taylor-Bernstein form to the rigorous treatment of infinite- dimensional path constraints in dynamic optimization remains unexplored. Addressing this gap, this paper introduces the Taylor-Bernstein form into the gradient-based dynamic optimization framework for the first time. By constructing a novel smooth upper bound function and designing a corresponding iterative algorithm, we achieve strict constraint satisfaction with improved efficiency. The main contributions are summarized as follows:

$\bullet$ A novel construction of smooth upper bound functions for path constraints and an associated inner approximation algorithm are proposed. The algorithm leverages the convex hull property of Bernstein polynomials and Log-Sum-Exp (LSE) smoothing to construct a rigorous yet differentiable upper bound function and yields a solution that strictly satisfies path constraints by iteratively solving approximation problems.

$\bullet$ Theoretical guarantees regarding the strictness of the upper bound and the convergence of the algorithm are rigorously established. Mathematical analysis proves that the constructed function serves as a strict upper bound for the original path constraints. Furthermore, the algorithm is proven to terminate within a finite number of iterations, converging to a KKT solution of the original problem with specified tolerances under strict feasibility.

$\bullet$ The tightness of the upper bound function and the computational efficiency of the algorithm are verified through numerical simulations. By mitigating the dependency effect, the proposed algorithm effectively reduces the constraint scale, thereby significantly improving computational efficiency.

The paper is structured as follows. Section 2 provides the problem formulation and necessary preliminaries. Section 3 details the construction of the smooth Taylor-Bernstein path constraint upper bound function and presents the associated gradient analysis. In Section 4, the inner approximation algorithm is proposed, and its theoretical properties are established. Section 5 presents numerical simulations to validate the effectiveness of the proposed method. Finally, Section 6 concludes the paper and outlines future research directions.

\section{Problem Statement and Preliminaries}
\subsection{Problem Statement}
Consider a dynamical system described by ordinary differential equations. We investigate the following path-constrained dynamic optimization problem, in which the control vectors have been parameterized\sloppy
\begin{equation} \label{eq1}
	\begin{split}
		\min_{u\in U} \quad & \Phi(x(t_f)) \\
		\mathrm{s.t.}\quad & \dot{x}(t) = f(x(t), u),\\
		& h(x(t), u) \leqslant 0, \quad t\in I = [t_0,t_f],\\
		& x(t_0) = x_0 ,\\
	\end{split}
\end{equation}
where $u=[u_1,u_2,\dots,u_{n_u}]^T$ denotes the parameterized finite-dimensional control vector, $U\subset \mathbb{R}^{n_u}$ denotes the feasible region of the control vector, which is a nonempty compact set. $x(t)=[x_{1}(t),\dots,x_{n_x}(t)]^{T}$ represents the state vector, which can be obtained by integrating the ODE system. $f: \mathbb{R}^{n_x} \times\mathbb{R}^{n_u}\to\mathbb{R}^{n_x}$ is the right-hand side term of the differential equation. $h: \mathbb{R}^{n_x} \times\mathbb{R}^{n_u}\to\mathbb{R}^{n_h}$ denotes the path constraint. $x_0$ is a given initial state. The objective function $\Phi: \mathbb{R}^{n_x}\to\mathbb{R}$ is the endpoint cost. We assume that the objective function and path constraint functions are sufficiently smooth.
\begin{assumption}[\cite{24}]
	 There exists a real number $K>0$, such that $\|f(x,u)\| \leqslant K (1+\|x\|)$.
\end{assumption}
According to Assumption 1, for every $u$, there exists a corresponding $x(t)$ that satisfies the differential equation of the dynamical system.
\subsection{Preliminaries}
In the algorithm design of this paper, interval analysis and Bernstein polynomials serve as our primary tools. This section provides a brief introduction to their concepts and properties. For a comprehensive treatment, see \cite{22,23}.

Let $\mathbb{IR}$ denote the set of compact real intervals. An interval $T = [\underline{T}, \overline{T}] \in \mathbb{IR}$ defines a set $\{t \in \mathbb{R} \mid \underline{T} \leqslant t \leqslant \overline{T}\}$. For a continuous function $h: \mathbb{R} \to \mathbb{R}$, an inclusion function $[h]: \mathbb{IR} \to \mathbb{IR}$ satisfies the property that for any sub-interval $T \subseteq \mathbb{IR}$, the range of the function is contained within the interval evaluation, i.e., $\{h(t) \mid t \in T\} \subseteq [h](T)$. While standard interval arithmetic guarantees inclusion, it often suffers from the dependency effect, leading to overestimation errors that scale linearly with the interval width \cite{36}. To mitigate this, we utilize the Bernstein form.

Bernstein polynomials provide a tighter enclosure for polynomial functions over a bounded interval. Consider a polynomial $p(\tau)$ of degree $n$ defined on the normalized interval $\tau \in [0, 1]$. The Bernstein basis polynomials of degree $r$ ($r \geqslant n$) are defined as
\begin{equation}\label{eq2}
	B_j^r(\tau)=\binom{r}{j}\tau^j(1-\tau)^{r-j},\quad j=0,1,\ldots,r,
\end{equation}
which satisfy the partition of unity property $\sum_{j=0}^{r}B_{j}^{r}(\tau) \equiv 1$ and non-negativity $B_{j}^{r}(\tau) \geqslant 0$. Consequently, the polynomial $p(\tau)$ can be uniquely represented in Bernstein form as
\begin{equation}\label{eq3}
	p(\tau)=\sum_{j=0}^rb_jB_j^r(\tau),
\end{equation}
where the Bernstein coefficients $b_j$ are derived from the coefficients $\beta_s$ of the original power basis form via the linear mapping
\begin{equation}\label{eq4}
	b_j=\sum_{s=0}^{\min(j,n)}\beta_s\frac{\binom{j}{s}}{\binom{r}{s}}.
\end{equation}
Leveraging the convex hull property of the basis functions, the range of $p(\tau)$ over $[0, 1]$ is strictly bounded by the minimum and maximum of its coefficients
\begin{equation} \label{eq5}
	p(\tau)\in\left[\min_{0\leqslant j\leqslant r}b_j,\max_{0\leqslant j\leqslant r}b_j\right],\quad\forall\tau\in[0,1].
\end{equation}
\begin{remark}
The envelope provided by \eqref{eq5} offers significant advantages over natural interval extensions, effectively mitigating interval dependence effects. As the interval width $\Delta(T)\rightarrow0$, the overestimation error of standard interval arithmetic decays linearly, whereas the Bernstein convex hull converges to the exact range of the polynomial at a quadratic rate \cite{32}. By leveraging this structural property, the Bernstein form provides a tighter, less conservative enclosure for polynomials than traditional interval methods. The corresponding visual comparison is illustrated in Fig. \ref{fig1}.
\end{remark}
\begin{figure}[htpb] 
	\centering 
	\includegraphics[height=8.5cm,width=8.5cm]{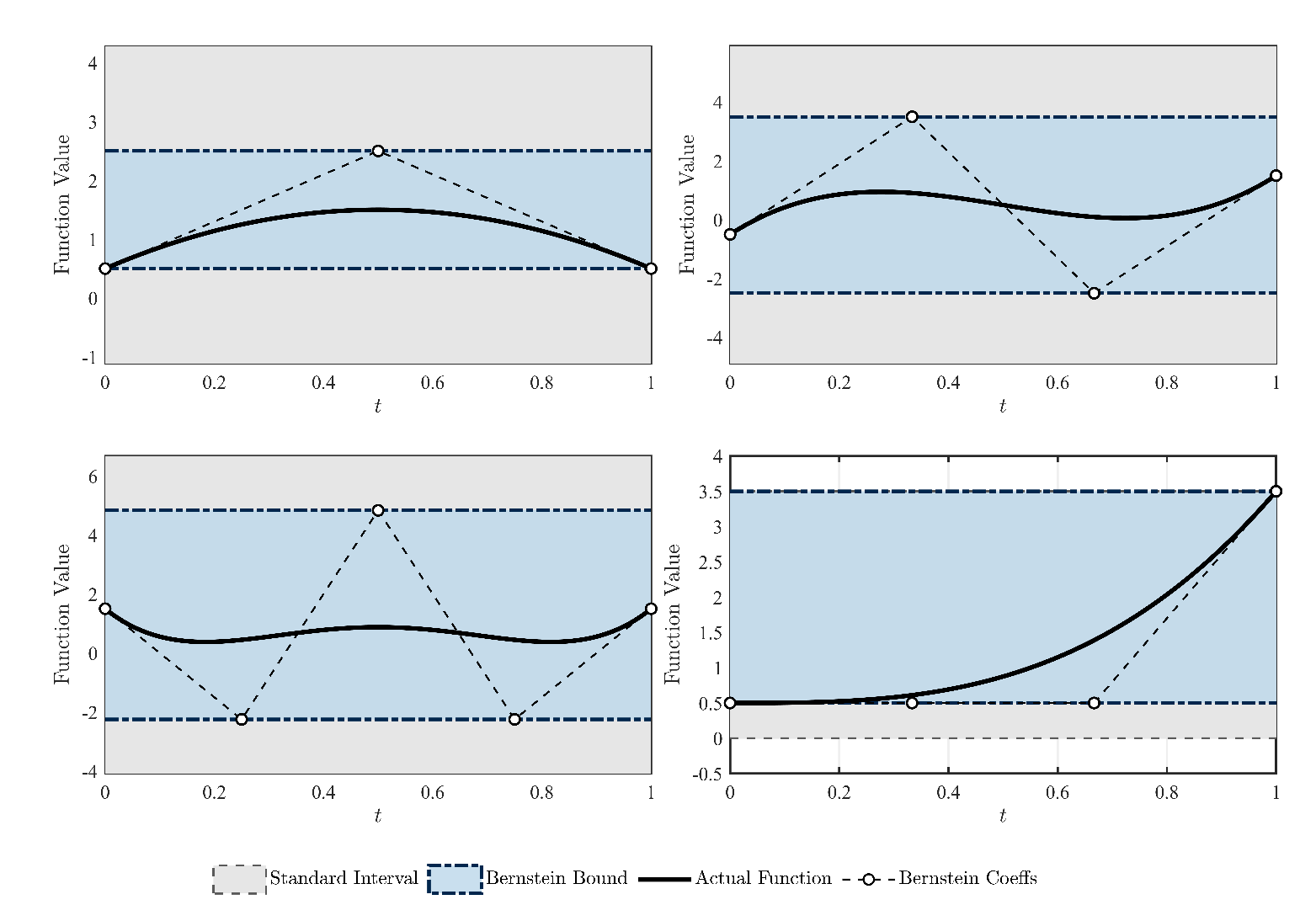} 
	\caption{Comparison of polynomial enclosures between standard interval arithmetic and the Bernstein form. This figure illustrates four different polynomial functions, where shaded areas represent the respective enclosures. In the first three cases, the bounds of standard interval arithmetic are omitted from the view as they significantly exceed the vertical axis limits due to severe overestimation. In contrast, the Bernstein form yields significantly tighter bounds by leveraging the convex hull property, demonstrating its capability to suppress the dependency effect and reduce conservatism.} 
	 \label{fig1}
\end{figure}
\section{Construction of Path Constraint Upper Bound Functions}
Utilizing Bernstein polynomials and interval analysis, this section constructs a smooth upper bound function for path constraints to facilitate subsequent algorithm design.
\subsection{Upper Bound Function Construction}
Let $\mathcal{\varGamma}$ denote the set of all closed subintervals of $[t_0, t_f]$. For any interval $T \in \mathcal{\varGamma}$, we perform a $q$th-order Taylor expansion of the path constraint function $h$ with respect to time $t$ at the midpoint $c(T)=\frac{1}{2}(\underline{T}+\overline{T})$ of the interval
\begin{equation}\label{eq6}
	\begin{aligned}
	h(u,t)&=\sum_{i=0}^{q-1}\frac{1}{i!}\frac{\partial^ih(u,c(T))}{\partial t^i}(t-c(T))^i\\
	&+\frac{1}{q!}\frac{\partial^qh(u,\xi)}{\partial t^q}(t-c(T))^q,\quad \forall t \in T,
	\end{aligned}
\end{equation}
where $\xi$ lies between $t$ and $c(T)$. For simplicity, we set
\begin{equation}
	p_{q-1}(u,t)=\sum_{i=0}^{q-1}a_i(u,c(T))(t-c(T))^i, 
\end{equation}
where $a_i(u,c(T))=\frac{1}{i!}\frac{\partial^ih(u,c(T))}{\partial t^i}$ denotes the coefficient of the Taylor polynomial, and $r_q(u,t)=\frac{1}{q!}\frac{\partial^qh(u,\xi)}{\partial t^q}(t-c(T))^q$.
Subsequently, we will separately address the polynomial part $p_{q-1}(u,t)$ and the remainder term part $r_q(u,t)$.

For the polynomial part $p_{q-1}(u,t)$, we utilize the convex hull property of Bernstein polynomials to construct a tight upper bound. Since the standard Bernstein basis functions are defined on the interval $[0,1]$, we transform the polynomial $p_{q-1}(u,t)$ from the original time interval $T$ onto the interval $[0,1]$. The affine transformation is defined as follows
\begin{equation}\label{eq7}
	\tau(t)=\frac{t-\underline{T}}{\Delta(T)}\Longrightarrow t(\tau)=\underline{T}+\tau \Delta(T),
\end{equation}
where $\Delta(T)$ denotes the length of the time interval $T$. By substituting \eqref{eq7} into $p_{q-1}(u,t)$, we obtain
\begin{equation} \label{eq8}
	\begin{aligned}
	p_{q-1}^*(\tau)&=\sum_{i=0}^{q-1}a_i(u,c(T))\left[\Delta(T)(\tau-\frac{1}{2})\right]^i
	\triangleq \sum_{l=0}^{q-1}\alpha_l(u,T)\tau^l,
	\end{aligned}
\end{equation}
where $\alpha_l(u,T)$ denotes the polynomial coefficient after the affine transformation. Next, we will transform $p_{q-1}^*(\tau)$ into the form of a $r$th-order Bernstein polynomial
\begin{equation}\label{eq9}
	p_{q-1}^*(\tau)=\sum_{j=0}^rb_j(u,T)\cdot B_j^r(\tau),
\end{equation}
where the Bernstein coefficient $b_j$ is calculated from coefficient $\alpha_l$ according to \eqref{eq4}
\begin{equation} \label{eq10}
 b_j(u,T)=\sum_{l=0}^{min(j,q-1)}\alpha_l(u,T)\cdot\frac{\binom{j}{l}}{\binom{r}{l}}.
\end{equation}

 Based on \eqref{eq5}, we can derive an upper bound for $p_{q-1}^*(\tau)$ over $\tau \in [0,1]$. Since $t$ and $\tau$ form a one-to-one correspondence, this upper bound also holds for 	$p_{q-1}(u,t)$ at $t \in T$
\begin{equation}\label{eq11}
	p_{q-1}(u,t)\leqslant\max_{j=0,\dots ,r}\{b_j(u,T)\},\quad\forall t\in T.
\end{equation}

For the remainder term $r_q(u,t)$, we utilize interval analysis techniques to construct the interval envelope $R_q(u,T)$ of the remainder term as \cite{17}. The interval envelope is expressed as 
\begin{equation}\label{eq12}
	R_q(u,T)=\frac{1}{q!}\cdot h^{(q)}(u,T)\cdot(T-c(T))^q,
\end{equation}
where $h^{(q)}(u,T)=[\underline{h^{(q)}}(u,T),\overline{h^{(q)}}(u,T)]$ denotes the interval envelope of the $q$th derivative of $h(u,T)$ at $t\in T$. For convenience in solving, we introduce a global constant interval $[\Lambda_L,\Lambda_U]$, which envelops the $q$th derivative over the entire optimization domain
\begin{equation}\label{eq13}
	\left\{\frac{\partial^qh(u,t)}{\partial t^q} {\bigg|} \forall(u,t)\in U\times I\right\}\subseteq[\Lambda_L,\Lambda_U].
\end{equation}

We can use differential inequalities or heuristic methods to compute $[\Lambda_L,\Lambda_U]$. For details, see \cite{17}. By substituting the global bounds $[\Lambda_L, \Lambda_U]$ into the remainder expression, we obtain a constant interval envelope $R_q(T)$ independent of $u$
\begin{equation}\label{eq14}
	R_q(T)=\frac{1}{q!}\cdot[\Lambda_L,\Lambda_U]\cdot\left[-\frac{\Delta(T)}{2},\frac{\Delta(T)}{2}\right]^q.
\end{equation}
According to interval arithmetic rules, the upper bound of $R_q(T)$ can be expressed as
\begin{equation}\label{eq15}
	\overline{R}_q(T)\triangleq\frac{1}{q!}\left(\frac{\Delta(T)}{2}\right)^qB_U,
\end{equation}
where $B_U$ is a precomputed constant defined as $B_U=\max\{|\Lambda_L|,|\Lambda_U|\}$ when $q$ is odd, and $B_U=\max\{0,\Lambda_U\}$ when $q$ is even.

Combining the polynomial bound \eqref{eq11} and the remainder bound \eqref{eq15}, we obtain a tight theoretical upper bound for the path constraint on interval $T$
\begin{equation}\label{eq16}
	h(u,t)\leqslant\max_{j=0,\dots,r}\left\{b_j(u,T)\right\}+\overline{R}_q(T).
\end{equation}

Although \eqref{eq16} establishes a strict upper bound, the non-differentiable max operator poses significant challenges for gradient-based NLP solvers. A direct approach involves imposing constraints on each Bernstein coefficient separately \cite{16,33,34}. However, this leads to an linear increase in the number of constraints when higher-order approximations or adaptive interval refinement are required. This results in an excessive computational burden during the optimization process. To effectively reduce the constraint scale while preserving differentiability, we construct a globally differentiable and strictly conservative upper bound using the LSE function \cite{25} defined as
\begin{equation}\label{eq17}
	\begin{aligned}
	H_{TB}(u,T)&\triangleq \frac{1}{\rho}\ln\left(\sum_{j=0}^r\exp(\rho\cdot b_j(u,T))\right)+\overline{R}_q(T),
	\end{aligned}
\end{equation}
where, for any $\rho> 0$, the deviation between the LSE smooth term and the actual maximum coefficient satisfies
\begin{equation}\label{eq18}
	\begin{aligned}
	\frac{1}{\rho}\ln\left(\sum_{j=0}^r\exp(\rho\cdot b_j(u,T))\right)-\max_{j=0,\ldots,r}\{b_j(u,T)\}\leqslant\frac{\ln(r+1)}{\rho}.
	\end{aligned}
\end{equation}

Since both $r$ and $\rho$ are predetermined constants, the right-hand side of \eqref{eq18} is a constant. We denote it as $\delta\triangleq\frac{\ln(r+1)}{\rho}$.  The rigorous relationship between the constructed smooth upper bound function and the original path constraint is given in the following theorem.

\begin{theorem}
For any smoothing parameter $\rho > 0$, polynomial degree $r \geqslant 1$, and time subinterval $T \subseteq I$, the smooth function $H_{TB}(u, T)$ defined in (17) serves as a strict upper bound for the path constraint function $h(u, t)$. Specifically, the overestimation error gap, defined as $E(u, T) \triangleq H_{TB}(u, T) - \max_{t \in T} h(u, t)$, satisfies
\begin{equation}\label{eq19}
	0<E(u,T)\leqslant \delta+O(\Delta(T)^2)+O(\Delta(T)^q).
\end{equation}
\end{theorem}
\begin{proof}
 Based on the path constraints  
\begin{equation}
 	h(u,t)=p_{q-1}(u,t)+r_q(u,t),
\end{equation}
we have
\begin{equation}\label{eq20}
	\begin{aligned}
	\max_{t\in T}h(u,t)&\leqslant\max_{t\in T}p_{q-1}(u,t)+\max_{t\in T}r_q(u,t)\\&\leqslant\max_j\{b_j\}+\overline{R}_q(T).
	\end{aligned}
\end{equation}

Substituting the definition of $H_{TB}$ from \eqref{eq17}, the error gap can be written as
\begin{equation}\label{eq21}
	\begin{aligned}
		E(u,T) &= H_{TB}(u,T) - \max_{t \in T} h(u,t) \\
		&\geqslant \frac{1}{\rho} \ln \left(\sum_{j=0}^{r} \exp(\rho \cdot b_j(u,T))\right)-\max_{j} \{b_j\} \\
		&> 0.
	\end{aligned}
\end{equation}

Thus, $E(u, T) > 0$ holds strictly, proving that $H_{TB}(u,T)$ is a strictly valid upper bound. Furthermore, to quantify the strictness of this boundary, we have
\begin{equation}\label{eq22}
	\begin{aligned}E(u,T)&=H_{TB}(u,T)-\max_{t\in T}h(u,t)\\&\leqslant \left(\max_j\{b_j\}+\delta+\overline{R}_q(T)\right)-\max_{t\in T}h(u,t)\\&=\delta+\left[(\max_j\{b_j\}+\overline{R}_q(T))-\max_{t\in T}h(u,t)\right],
	\end{aligned}
\end{equation}
where the approximation error term $(\max_j\{b_j\}+\overline{R}_q(T))-\max_{t\in T}h(u,t)$  is dominated by the convergence rate of the Bernstein polynomial ($O(\Delta(T)^2)$) and the Taylor remainder ($O(\Delta(T)^q)$). Consequently, we arrive at the final bound
\begin{equation}\label{eq23}
	E(u,T)\leqslant\delta+O(\Delta(T)^2)+O(\Delta(T)^q).
\end{equation}
\end{proof}
\begin{remark}
	From \eqref{eq23}, it is observed that the smoothing parameter $\rho$ introduces a distinct component to the overestimation error. Unlike the approximation error in Taylor's remainder and Bernstein's approximation, this term manifests as a systematic, strictly positive safety margin that does not vanish as the interval width decreases. Therefore, this method ensures solutions remain within a slightly contracted feasible region. When $\rho\rightarrow\infty$, this upper bound function converges to the exact envelope of the original path constraints.
\end{remark}
\subsection{Gradient Analysis}
To apply gradient-based NLP solvers, it is necessary to obtain sensitivity information regarding the smooth upper bound function $H_{\mathrm{TB}}(u,T)$ with respect to the control input $u$. Since the Taylor remainder $\overline{R}_q(T)$  is constant with respect to $u$, the gradient depends directly on the sensitivity of the Bernstein coefficients to the control input.

Differentiating \eqref{eq17} with respect to $u$ yields
\begin{equation}\label{eq24}
	\nabla_uH_{TB}(u,T)=\sum_{j=0}^rw_j(u)\cdot\nabla_ub_j(u,T),
\end{equation}
where $w_{j}(u)$ represents the weighting coefficient derived from the derivative of the LSE term, defined as
\begin{equation}\label{eq25}
	w_j(u)\triangleq\frac{\exp(\rho\cdot b_j)}{\sum_{l=0}^k\exp(\rho\cdot b_l)}.
\end{equation}

Calculating directly from \eqref{eq10} would involve multiple nested sums, potentially incurring excessive computational cost. Since the mapping from the time derivatives of $h(u,t)$ to the Bernstein coefficients is linear and time-invariant, we reformulate the gradient computation into a compact matrix form. Let $D(u)\triangleq[h(u,c(T)),\frac{\partial h}{\partial t},\ldots,\frac{\partial^{q-1}h}{\partial t^{q-1}}]^T$ denote the vector of time derivatives of $h(u,t)$ at the midpoint $c(T)$ of the interval. The Bernstein coefficient vector $b=[b_0,\ldots,b_r]^T$ can be expressed as the following linear transformation of $D(u)$
\begin{equation}\label{eq26}
	b=M(T)\cdot D(u),
\end{equation}
where $M(T) \in \mathbb{R}^{(r+1) \times q}$ is a constant transformation matrix. By aggregating terms from \eqref{eq10} and \eqref{eq12}, the elements of $M(T)$ are explicitly given by
\begin{equation}\label{eq27}
	M_{j,i}=\frac{\Delta(T)^i}{i!}\sum_{l=0}^{\min(j,i)}\frac{\binom{j}{l}}{\binom{r}{l}}\binom{i}{l}\left(-\frac{1}{2}\right)^{i-l},
\end{equation}
for $j=0,\dots,r$ and $i=0,\dots,q-1$. Consequently, the total gradient is computed efficiently via matrix-vector multiplication
\begin{equation}\label{eq28}
	\nabla_uH_{TB}(u,T)=W^T\cdot M(T)\cdot\nabla_uD(u),
\end{equation}
where $W = [w_0, \dots, w_r]^T$ and $\nabla_uD(u)$ represents the Jacobian of the time derivatives with respect to $u$. The matrix $M(T)$ consists of constant coefficients that can be precomputed during the algorithm's initialization phase to reduce computational overhead during optimization.
\begin{remark}
When the smoothing parameter $\rho$ is large, directly computing the exponential terms in \eqref{eq17} and \eqref{eq25} risks floating-point overflow. To ensure numerical stability, we employ the LSE shifting technique \cite{29} by normalizing the exponential terms prior to computation. This involves subtracting the maximum exponent value $\rho b_{\max}$ from each exponent, ensuring that the arguments of all exponential functions remain non-positive and thereby guaranteeing robust numerical stability in both function evaluation and gradient computation regardless of the magnitude of $\rho$.
\end{remark}
\section{Algorithm Design and Convergence Analysis}
Based on the smooth upper bound function $H_{TB}(u,T)$ constructed in Section 3, we can transform the original problem \eqref{eq1} into the following approximation problem
\begin{equation} \label{eq29}
	\begin{split}
		\min_{u\in U} \quad & \Phi(x(t_f)) \\
		\mathrm{s.t.}\quad & \dot{x}(t) = f(x(t), u),\\
		& H_{TB}(u,T_{i})\leqslant0,\quad i=1,\ldots,n_H, \\
		& x(t_0)=x_0(u),\\
	\end{split}
\end{equation}
where $n_H$ denotes the number of subintervals into which the interval is subdivided, and $H_{TB}(u,T_{i})$ represents the upper bound function of path constraints on the subinterval $T_{i}$. Since $h(u, t) \leqslant H_{TB}(u, T_{i})$ for all $t \in T_{i}$, any feasible solution to the approximation problem is also a feasible solution to the original problem.
\subsection{Algorithm Description}
We propose a systematic algorithm that iteratively solves the sequence of inner approximation problems \eqref{eq29}, with its procedure summarized in Algorithm \ref{alg1} and Fig. \ref{fig2}. To accelerate convergence to a KKT solution in Algorithm \ref{alg1}, we employ an adaptive subdivision strategy. Specifically, based on the error bounds derived from Theorem 1, we estimate the number of subintervals required to meet the tolerance in a single iteration as follows
\begin{equation}\label{eq30}
	N_m^k=\max\left\{2, \left\lceil\frac{\Delta(T_m^k)}{\sqrt[q]{2^q\cdot q!\cdot(\epsilon_{act}-\delta)/B_U}}\right\rceil\right\}.
\end{equation}

In Case 1, employing this subdivision strategy ensures that the interval width is rapidly narrowed to meet the tolerance requirements.

\begin{figure}[htpb] 
	\centering 
	\includegraphics[height=8.5cm,width=7cm]{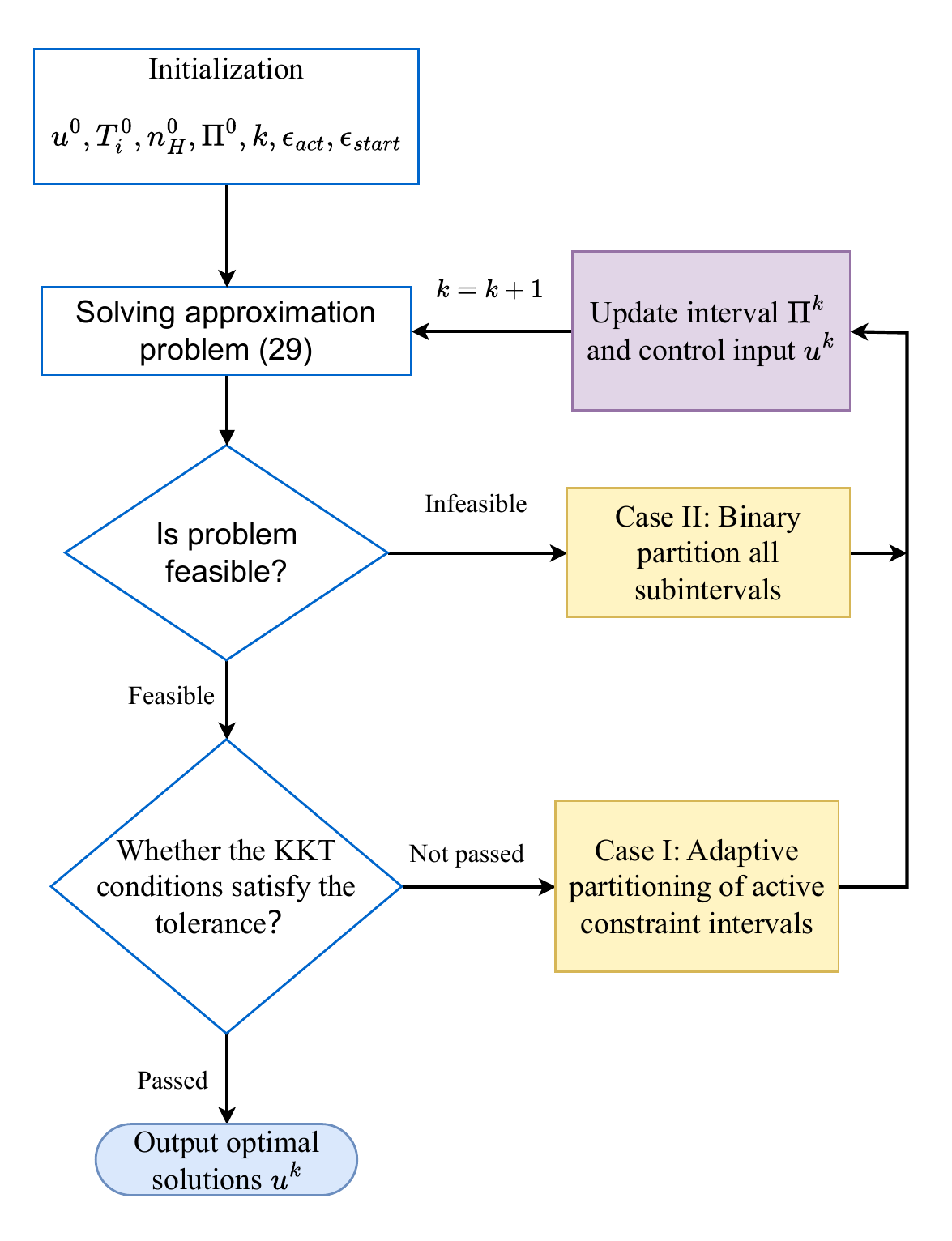} 
	\caption{Algorithm flowchart} 
	\label{fig2} 
\end{figure}

\begin{remark}
It is noteworthy that \eqref{eq30} employs $(\epsilon_{act} - \delta)$ as the subdivision tolerance criterion, which differs significantly from the strategy in \cite{17} that directly uses $\epsilon_{act}$. This is because the smoothing error in \cite{17} naturally vanishes as the interval width decreases, whereas the bias $\delta$ introduced by the LSE method in this paper is constant and independent of the interval width. To ensure the feasibility of the subdivision process, the smoothing parameter $\rho$ must be selected such that $\delta < \epsilon_{act}$ is strictly satisfied. Therefore, this inherent bias must be preemptively deducted from the total tolerance budget to prevent the algorithm from failing to converge.
\end{remark}

\begin{algorithm}[htb]
	\caption{Dynamic optimization algorithm strictly satisfying path constraints}
	\label{alg1}
	\begin{algorithmic}[1] 
		\Require 
		 NLP solver initial guess $u^0$;
		 number of subintervals $n_H^0$;
		 initial subinterval $T_i^0, i=1,\dots,n_H^0$;
		 set of subintervals $\Pi^{0}=\{T_{i}^{0},i=1,\ldots,n_{H}^{0}\}$;
		 iteration count $k$;
		 optimality tolerances $\varepsilon_{stat},\varepsilon_{act} > 0$.
		
		\Statex 
		
		\Repeat 
		\State Solve approximation problem (29) based on initial subintervals $T_i^0$
		\If{the problem is feasible}
		\State Set $u_k$ to the locally optimal solution. Obtain $n_u$ linearly independent gradients $\nabla_uH_\mathrm{TB}(u^k,T_m^k)$ and multipliers $\lambda_m^k$. Define subintervals associated with gradients: $\Pi_{grad}^k=\{\mathcal{T}_m^k,m=1,\ldots,n_u\}$. Set active subintervals $\Pi_{act}^{k}=\begin{Bmatrix}\mathcal{T}_{m}^{k}|H_{{TB}}(u^k,T_m^k)=0,\mathcal{T}_{m}^{k}\in\Pi_{grad}^{k}\end{Bmatrix}$. Set inactive subintervals $\Pi_{inact}^k=\Pi^k-\Pi_{act}^k$.
		
		\State \textbf{if} $u_k$ satisfies the following conditions:
		\begin{equation}\label{eq31}
			\|\nabla_{u}\Phi(x(t_{f},u^{k})) + \sum_{m=1}^{n_u}\lambda_{m}^{k}\nabla_{u}h(u^{k},c(T_m^k))\| \leqslant \epsilon_{stat}
		\end{equation}
		\begin{equation} \label{eq32}
			\lambda_m^k h(u^{k},c(T_m^k)) \in [-\lambda_m^k \epsilon_{act}, 0], m=1,\ldots,n_u,
		\end{equation}
		\begin{equation}\label{eq33}
			h(u^k,t)\leqslant0,\forall t\in I.
		\end{equation}
	\State \hspace{\algorithmicindent} \textbf{Terminate}
	
	\State \textbf{else} \Comment{(Case 1)}
	\State \hspace{\algorithmicindent} Divide each active subinterval $T_m^{k}$ in $\Pi_{act}^k$ 
	\State \hspace{\algorithmicindent}
	into $N_m^k$ equal parts.
	\State \hspace{\algorithmicindent} $\Pi^{k+1}\leftarrow\Pi_{act}^{\prime}\cup\Pi_{inact}^{k}$ and  $u^{k+1}\leftarrow u^{k}$.

		\Else \Comment{(Case 2)}
		\State Binary partition all subintervals $T_i^{k}$.
		\State $\Pi^{k+1}\leftarrow\Pi^{k}$ and $u^{k+1}\leftarrow u^{k}$.
		\EndIf
		\State Set $n_H^{k+1}\leftarrow|\Pi^{k+1}|$ and $k \leftarrow k + 1$.
		\Until{convergence}
	\end{algorithmic}
\end{algorithm}
\subsection{Convergence Analysis}
To establish finite convergence, we make the following assumptions
\begin{assumption}[\cite{17}]
	 For problem \eqref{eq1}, there exists a Slater control $u_s$ and a constant $\gamma > 0$ such that $\max_{t \in I} h(u_s, t) \leqslant -\gamma$.
\end{assumption}
\begin{assumption}[\cite{9}]
	 The KKT multipliers are nonnegative and uniformly bounded throughout all iterations.
\end{assumption}
\begin{theorem}
	Under Assumptions 1-3, if the smoothing parameter $\rho$ is sufficiently large such that $\delta < \epsilon_{act}$, Algorithm 1 terminates in a finite number of iterations with a solution satisfying the approximate KKT conditions of the original problem with tolerances $\epsilon_{stat}, \epsilon_{act} > 0$.
\end{theorem}
\begin{proof}
As shown in Algorithm 1, we only need to consider two cases: Case 1 and Case 2. To rule out the possibility of these two cases recurring indefinitely, we employ the proof approach from \cite{17}.
	
\textit{Excluding infinite occurrence of Case 1}: If Case 1 occurs infinitely, it implies that the active subinterval width $\Delta(T)\rightarrow0$. There exists a subsequence $\{k_r\}$ such that feasible solutions $u^{k_r}$ can be found for \eqref{eq29}, but none of these $u^{k_r}$ satisfy the KKT conditions of the original problem \eqref{eq1}. By Assumption 3, the multiplier sequence $\{\lambda^{k_r}\}$ is contained within the compact set $\varLambda$. Since $U\times\varLambda\times I$ is a compact set, the sequence $\{u^{k_r},\lambda^{k_r},c(T^{k_r})\}$ has a subsequence $\{\overline{k_r}\}$ that converges to the limit point $(u^*,\lambda^*,t^*)$. It remains to show that this limit point satisfies the termination conditions \eqref{eq31}--\eqref{eq33}.
	
First, any feasible solution to \eqref{eq29} is also a feasible solution to \eqref{eq1}, and condition \eqref{eq33} is always satisfied.
	
Second, to prove that the stability condition \eqref{eq31} holds, the key is to demonstrate that $\nabla_uH_\mathrm{TB}(u^{k_r},T_m^{k_r})$ converges to $\nabla_{u}h(u^{k_r},c(T_m^{k_r}))$. Based on the previously obtained gradient \eqref{eq28}, we consider the limits for $\nabla_ub_j(u^{k_r},T_m^{k_r})$ and $w_j(u)$ separately. As the interval width $\Delta(T)\rightarrow0$, the interval $T_m^{k_r}$ contracts to its midpoint $c(T_m^{k_r})$. According to \eqref{eq11}, when $t\rightarrow c(T_m^{k_r})$, all terms $(t-c(T_m^{k_r}))^i$ for $i\geqslant1$ vanish, so we have 
	\begin{equation}\label{eq34}
		\lim_{\Delta(T)\to0}p_{q-1}(u^{k_r},t)=a_0(u^{k_r},c(T_m^{k_r})=h(u^{k_r},c(T_m^{k_r})).
	\end{equation}
	Since $p_{q-1}^*(\tau)$ is an affine transformation of $p_{q-1}$, and by leveraging the range enclosure property of the Bernstein form, we have $\lim_{\Delta(T)\to0}b_j(u^{k_r},T_m^{k_r})=h(u^{k_r},c(T_m^{k_r}))$. As $h(u,t)$ is sufficiently smooth, we obtain
	\begin{equation}\label{eq35}
		\begin{aligned}
		\lim_{\Delta(T)\to0}\nabla_ub_j(u^{k_r},T_m^{k_r})=\nabla_u(\lim_{\Delta(T)\to0}b_j(u^{k_r},T_m^{k_r}))=\nabla_uh(u^{k_r},c(T_m^{k_r})).
		\end{aligned}
	\end{equation}
	At the same time, for the limit of $w_j(u)$, we have
	\begin{equation}\label{eq36}
		\begin{aligned}
		\lim_{\Delta(T)\to0}w_j(u)=\frac{\exp(\rho\cdot h(u^{k_r},c(T_m^{k_r}))}{(r+1)\cdot\exp(\rho\cdot h(u^{k_r},c(T_m^{k_r})))}=\frac{1}{r+1}.
	    \end{aligned}
	\end{equation}
	Combining \eqref{eq35} and \eqref{eq36}, we obtain the limit of the total gradient as follows
	\begin{equation}\label{eq37}
		\begin{aligned}
		\lim_{\Delta(T)\to0}\nabla_uH_{TB}(u^{k_r},T_m^{k_r})&=\sum_{j=0}^r(\lim_{\Delta(T)\to0}w_j)\cdot(\lim_{\Delta(T)\to0}\nabla_ub_j(u^{k_r},T_m^{k_r}))\\&=\nabla_uh(u^{k_r},c(T_m^{k_r})).
		\end{aligned}
	\end{equation}
	
	Therefore, when $k_r\rightarrow\infty$, the stability condition in \eqref{eq29} converges to \eqref{eq31}. There exists a $k_r^{d_1}$ such that when $\forall k_r>k_r^{d_1}$, condition \eqref{eq31} is satisfied.
	
	Finally, we need to prove that condition \eqref{eq32} holds. For inactive constraints, condition \eqref{eq32} is clearly satisfied. For any active constraint $T_m^{k_r}\in\Pi_{act}^{k_r}$, we need to prove that $h(u^{k_r},c(T_m^{k_r}))$ converges to 0. As $\Delta(T)\rightarrow0$, by Theorem 1 we have
	\begin{equation}\label{eq38}
		\lim_{\Delta(T)\to0}\left|H_{TB}(u^{k_r},T_m^{k_r})-h(u^{k_r},c(T_m^{k_r}))\right|=\delta.
	\end{equation}
	
   Due to the strict enforcement of the active subinterval path constraint, we have $H_{TB}(u^{k_r},T_m^{k_r})=0$. Consequently, $h(u^{k_r},c(T_m^{k_r}))=-\delta$. To satisfy the complementarity tolerance condition \eqref{eq32}, which requires $h \in [-\epsilon_{act}, 0]$, the smoothing parameter $\rho$ must be selected sufficiently large such that the bias satisfies $\delta < \epsilon_{act}$. Under this condition, the limit value $-\delta$ falls strictly within the tolerance range. Hence, we can conclude that there exists an index $k_r^{d_2}$ such that for all $k_r > k_r^{d_2}$, condition \eqref{eq32} is satisfied.
	
	In summary, if Case 1 occurs infinitely, then for all $k_r>\max(k_r^{d1},k_r^{d2})$, termination conditions \eqref{eq31}--\eqref{eq33} will be simultaneously satisfied. This implies that the algorithm terminates at iteration $k_r$, which contradicts the assumption. Therefore, Case 1 cannot recur indefinitely.
	
	\textit{Excluding infinite occurrence of Case 2}: Assume that Case 2 recurs indefinitely, implying that there exists an infinite subsequence $\{k_l\}$ of iteration indices such that the feasible region of the approximate problem \eqref{eq29} is empty. According to the update rule of Algorithm 1, every occurrence of Case 2 results in a bisection of all subintervals. Consequently, at the $k_l$-th iteration, the maximum width of the subintervals satisfies
	\begin{equation}\label{eq39}
		\max_{T\in\Pi^{k_l}}\Delta(T)\leqslant\frac{\Delta(I)}{n_H^0\cdot2^l}.
	\end{equation}
	
	According to Assumption 2, there exists a Slater point $u_s$ and a constant $\gamma > 0$ such that the original path constraint is strictly satisfied
	\begin{equation}\label{eq40}
		\max_{t\in I}h(u_s,t)=-\gamma<0.
	\end{equation}
	
	From Theorem 1, the upper bound function $H_{TB}(u,T)$ accounts for the  Bernstein approximation error, a Taylor remainder, and a constant smoothing bias $\delta$. Therefore, there exists a constant $M > 0$ such that the overestimation error is bounded by
	\begin{equation}\label{eq41}
		H_{TB}(u_s,T)-\max_{t\in T}h(u_s,t)\leqslant \delta+ M\cdot(\Delta(T))^q.
	\end{equation}
	
	Combining these relations, we examine the feasibility of the Slater point $u_s$ for the approximate problem \eqref{eq29}. For any subinterval $T_i^{k_l} \in \Pi^{k_l}$, we have
	\begin{equation}\label{eq42}
		\begin{aligned}
			H_{TB}(u_s,T_i^{k_l})&\leqslant\max_{t\in T_{i}^{k_{l}}}h(u_{s},t)+ \delta+ M\cdot(\Delta(T_{i}^{k_{l}}))^{q}\\
			&\leqslant-\gamma+\delta+M\cdot\left(\frac{\Delta(I)}{n_H^0\cdot2^l}\right)^q.
		\end{aligned}
	\end{equation}
	
	We assume the smoothing parameter $\rho$ is chosen sufficiently large such that the bias is strictly smaller than the Slater margin, i.e., $\delta < \gamma$. Consequently, the constant term $-\gamma + \delta$ is strictly negative. Since the interval-dependent term converges to 0 as $l \to \infty$, there exists an integer $L$ such that for all $l > L$
    \begin{equation}\label{eq43}
	M\cdot\left(\frac{\Delta(I)}{n_H^0\cdot2^l}\right)^q<\gamma-\delta.
	\end{equation}
	
	This implies that for all $l > L$ and for all subintervals $T_i^{k_l} \in \Pi^{k_l}$, the condition $H_{TB}(u_s, T_i^{k_l}) < 0$ holds. Therefore, $u_s$ becomes a feasible solution for \eqref{eq29} when $l$ is sufficiently large, contradicting the assumption that the feasible region of \eqref{eq29} is empty. Thus, Case 2 cannot recur indefinitely.
	
    Since the infinite recurrence of both Case 1 and Case 2 has been ruled out, the algorithm is guaranteed to terminate in a finite number of iterations.
\end{proof}

\section{Simulation Examples}
In this section, the numerical efficacy of the proposed Taylor-Bernstein form inner approximation algorithm is evaluated against the Taylor model method \cite{17} using three benchmark problems adopted from \cite{26,27,35}. Specifically, Example 1 involves the path constraint $h(t) = -x_1(t) - 0.4 \leqslant 0$ with a remainder upper bound constant $B_U = 260$; Example 2 entails $h(t) = x_2(t)+ 0.5-8(t - 0.5)^2 \leqslant 0$ with $B_U = 33$; and Example 3 is subject to two constraints: $h_1(t) = -x_3(t) \leqslant 0$ ($B_U = 750$) and $h_2(t) = -x_2(t) - 0.8 \leqslant 0$ ($B_U = 20$). Detailed mathematical formulations of these examples are provided in \ref{app1}. To ensure a fair comparison, all numerical simulations were performed on an Intel Core i5-10300H CPU (2.50 GHz) using MATLAB R2024b, utilizing the $ode45$ solver and the $sqp$ algorithm within $fmincon$. Uniform parameter settings were applied across all examples: Taylor expansion order $q=3$, Bernstein polynomial order $r=2$, smoothing parameter $\rho=1500$, and tolerances for the KKT conditions fixed at $10^{-3}$. Additionally, the number of CVP segments was set to 30 for Examples 1 and 3, and 20 for Example 2.

Figs. \ref{fig3}--\ref{fig5} illustrate the comprehensive simulation results for the three benchmark problems. As summarized in Table \ref{tab1}, the proposed Taylor-Bernstein algorithm significantly reduces the scale of discretized constraints without compromising solution optimality, resulting in a reduction in total CPU time of approximately 30\% to 40\%. This computational efficiency is explicitly corroborated by the KKT stationarity convergence profiles in Figs. \ref{fig3c}--\ref{fig5c}, which consistently demonstrate the improved convergence rate of the proposed method compared to the Taylor Model approach. Furthermore, regarding constraint satisfaction, the path constraint trajectories in Figs. \ref{fig3b}--\ref{fig5b} confirm that the proposed method ensures strict adherence to path constraints throughout the entire time horizon. 

\begin{table*}[htbp]
	\caption{Comparison of computational performance between Taylor model and Taylor-Bernstein methods across benchmark examples. (Examples 1 and 2 consist of a single path constraint, whereas Example 3 involves two distinct path constraints.)}
	\centering
	
	\renewcommand{\arraystretch}{1.5}
	
	\begin{tabular}{c|ccccc}
		\hline
		
		Case Study & Algorithm & Iterations & Number of Constraints &  CPU time(s) & Optimal Cost Function \\
		\hline

		\multirow{2}{*}{Example 1} 
		& Taylor model & 7 & 293 & 5.30 & 2.96  \\
		& Taylor-Bernstein & 3 & 87 & 3.21 & 2.96  \\
		\hline

		\multirow{2}{*}{Example 2} 
		& Taylor model & 5 & 97 & 5.93 & 0.17  \\
		& Taylor-Bernstein & 3 & 44 & 3.84 & 0.17  \\
		\hline

		\multirow{2}{*}{Example 3} 
		& Taylor model & 7 & 41+669 & 9.34 & 0.033  \\
		& Taylor-Bernstein & 4 & 40+104 & 6.19 & 0.033  \\
		\hline
	\end{tabular}
	\label{tab1}
\end{table*}

\begin{figure*}[htpb] 
	\centering
	\begin{subfigure}[b]{0.328\linewidth}
		\includegraphics[width=\linewidth]{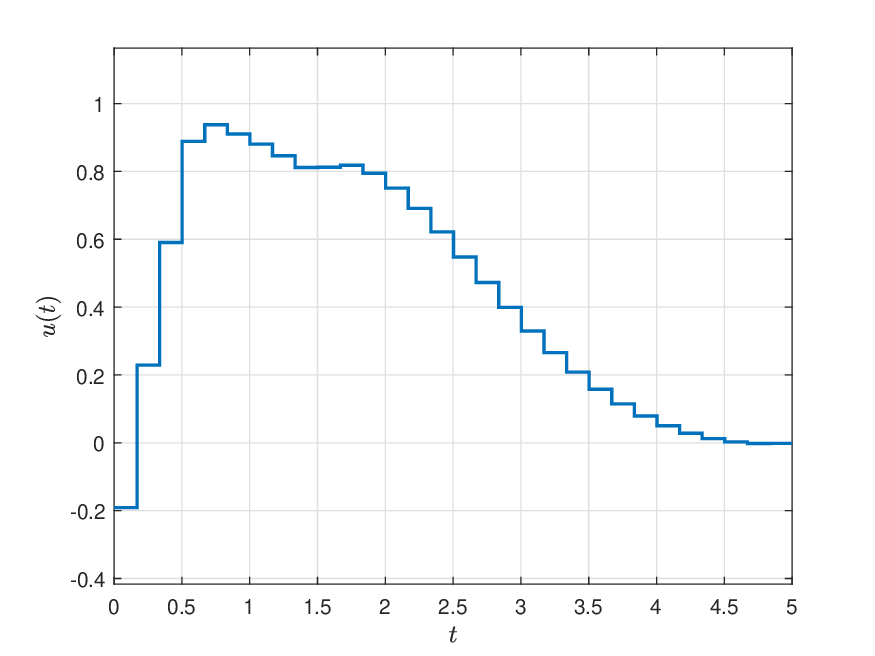}
		\caption{Control input}
		\label{fig3a}
	\end{subfigure}
	\hfill
	\begin{subfigure}[b]{0.328\linewidth}
		\includegraphics[width=\linewidth]{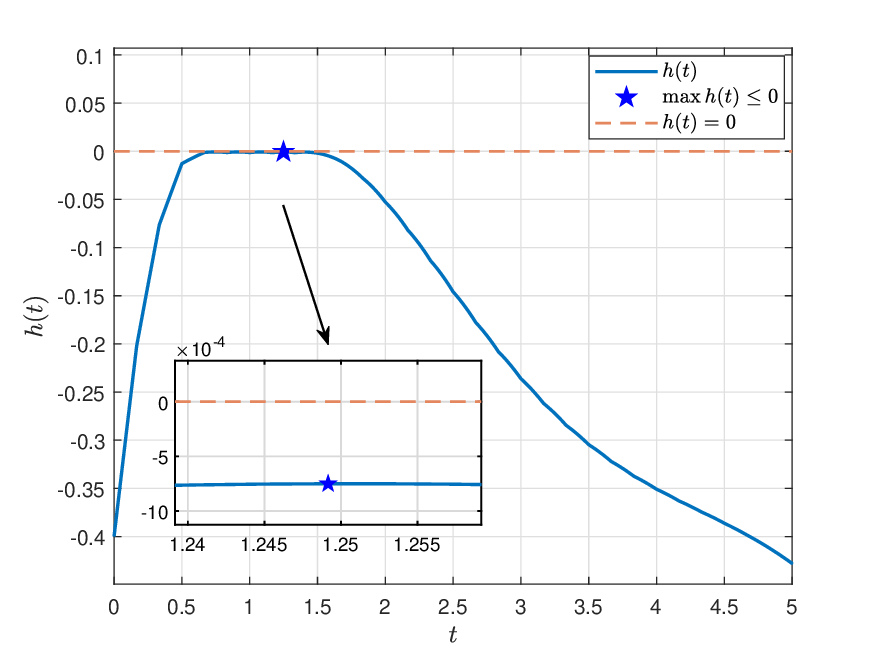}
		\caption{Path constraint }
		\label{fig3b}
	\end{subfigure}
	\hfill
	\begin{subfigure}[b]{0.328\linewidth}
		\includegraphics[width=\linewidth]{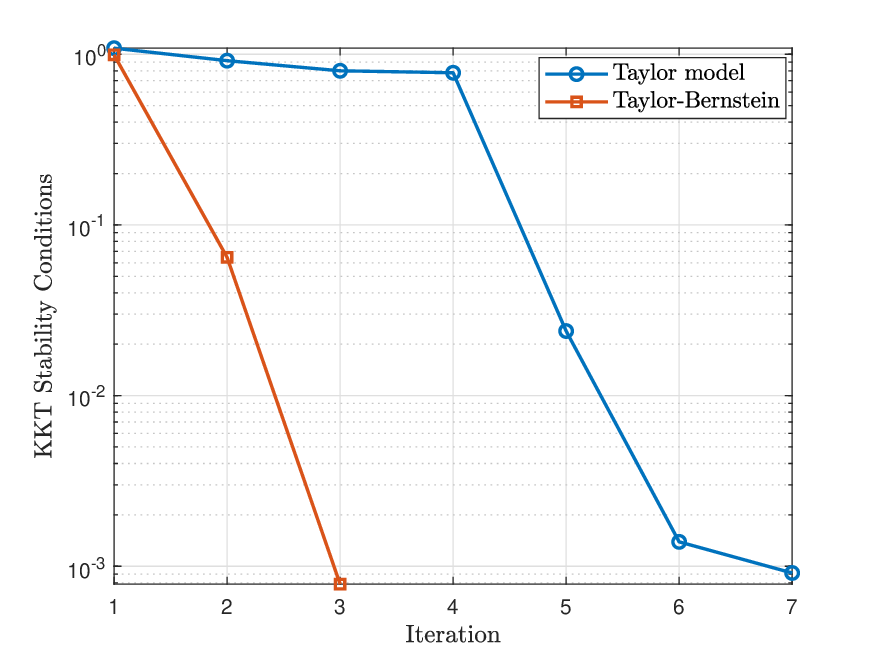}
		\caption{KKT stationarity error}
		\label{fig3c}
	\end{subfigure}
	\caption{Simulation results of Example 1.}
	\label{fig3}
\end{figure*}

\begin{figure*}[htpb] 
	\centering
	\begin{subfigure}[b]{0.328\linewidth} 
		\includegraphics[width=\linewidth]{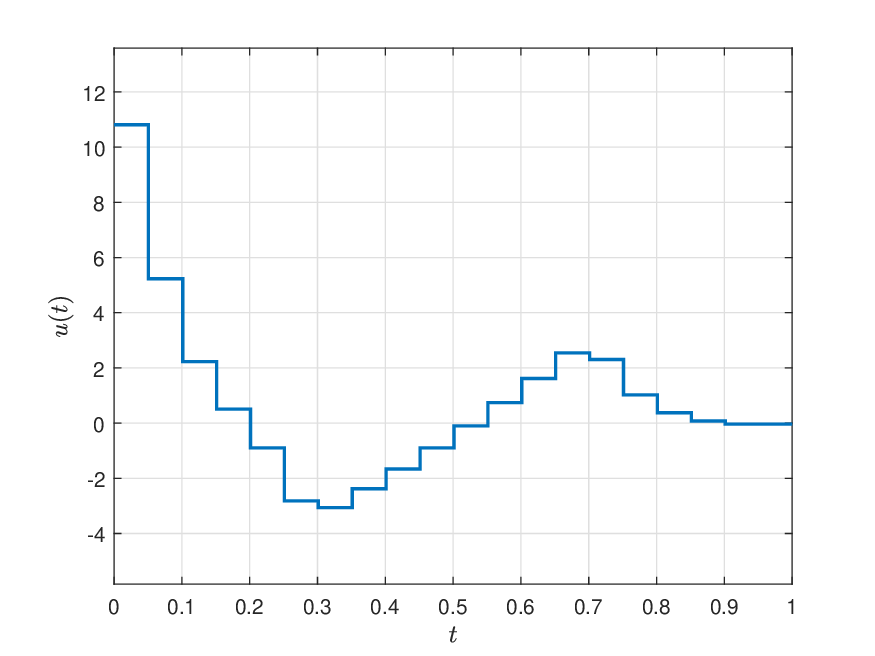}
		\caption{Control input}
		\label{fig4a}
	\end{subfigure}
	\hfill
	\begin{subfigure}[b]{0.328\linewidth}
		\includegraphics[width=\linewidth]{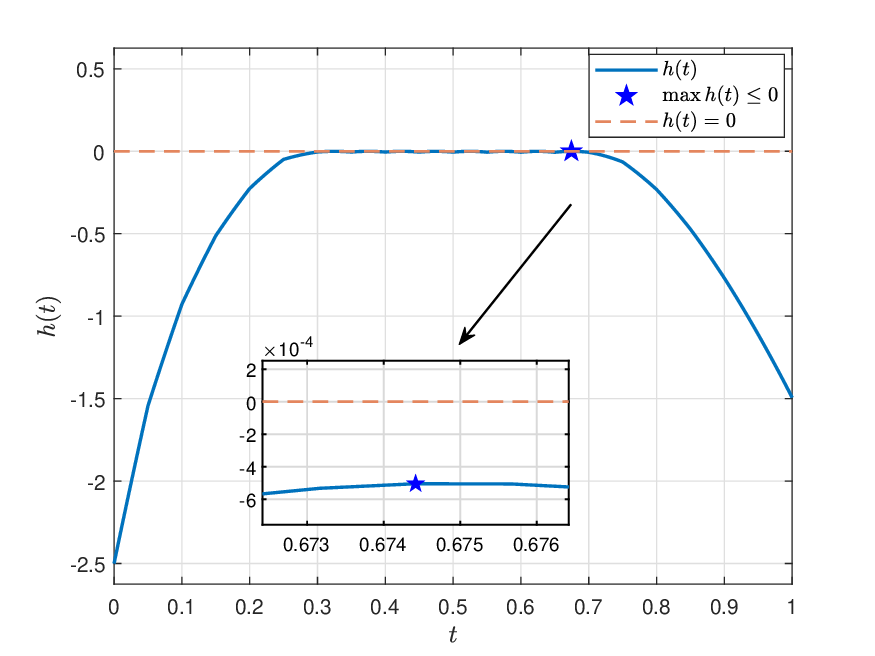}
		\caption{Path constraint }
		\label{fig4b}
	\end{subfigure}
	\hfill
	\begin{subfigure}[b]{0.328\linewidth}
		\includegraphics[width=\linewidth]{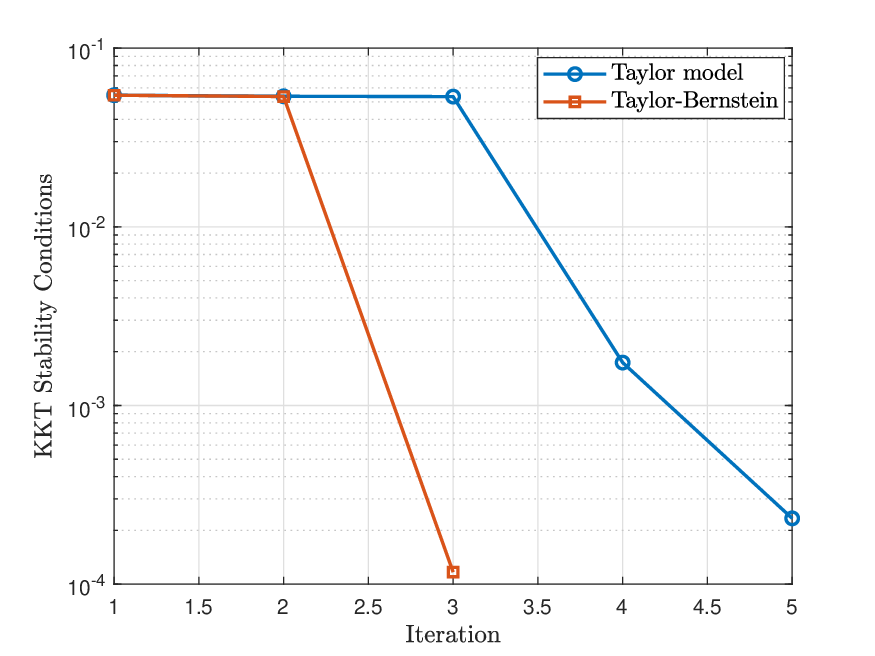}
		\caption{KKT stationarity error}
		\label{fig4c}
	\end{subfigure}
	\caption{Simulation results of Example 2.}
	\label{fig4}
\end{figure*}

\begin{figure*}[htp] 
	\centering
	\begin{subfigure}[b]{0.328\linewidth} 
		\includegraphics[width=\linewidth]{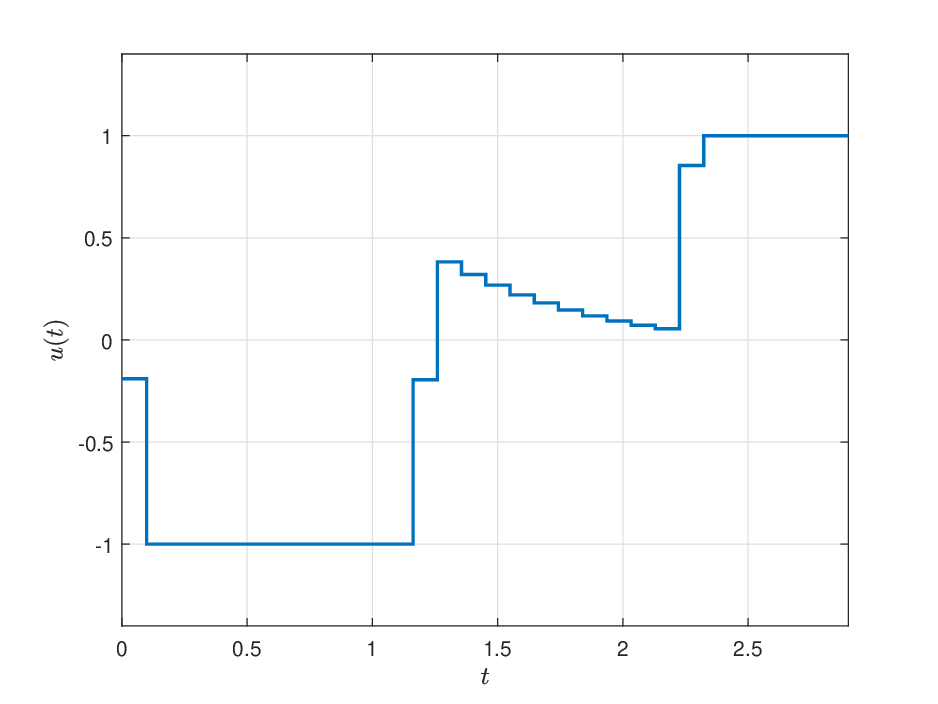}
		\caption{Control input}
		\label{fig5a}
	\end{subfigure}
	\hfill
	\begin{subfigure}[b]{0.328\linewidth}
		\includegraphics[width=\linewidth]{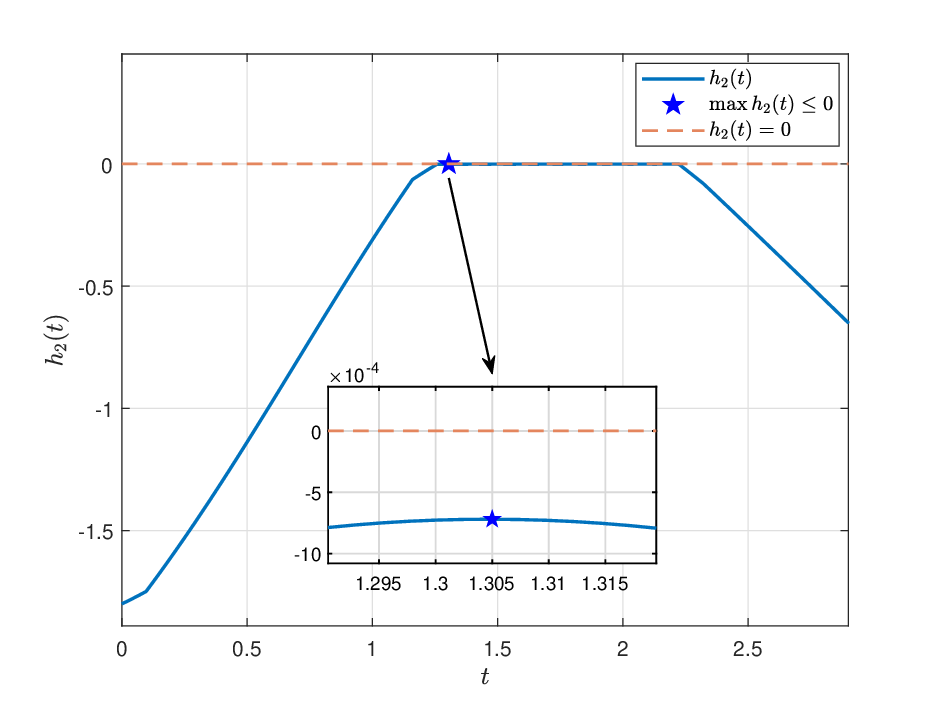}
		\caption{Path constraint }
		\label{fig5b}
	\end{subfigure}
	\hfill
	\begin{subfigure}[b]{0.328\linewidth}
		\includegraphics[width=\linewidth]{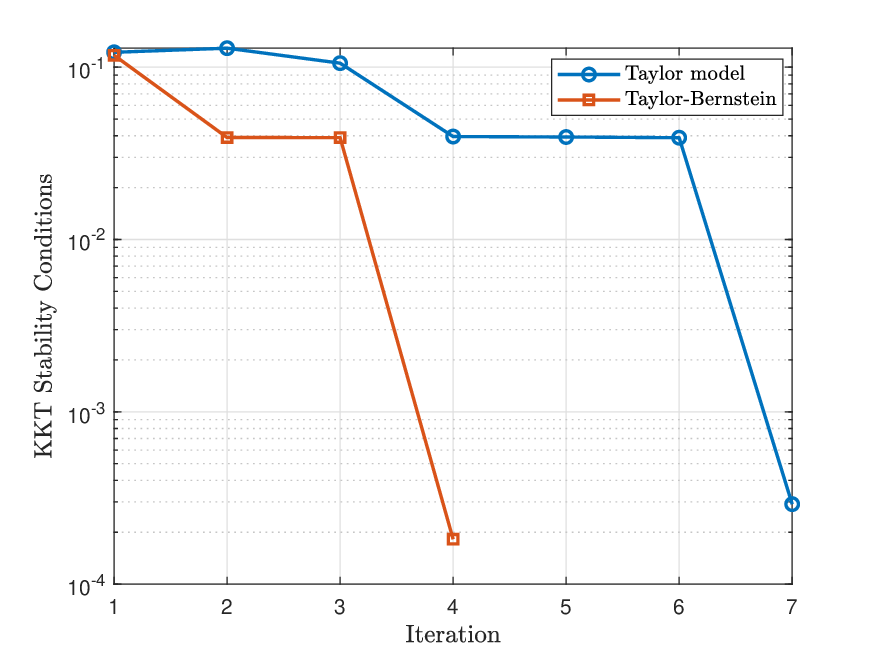}
		\caption{KKT stationarity error}
		\label{fig5c}
	\end{subfigure}
	\caption{Simulation results of Example 3.}
	\label{fig5}
\end{figure*}

To further elucidate the underlying reasons for the superior computational efficiency of the Taylor-Bernstein form inner approximation algorithm, we conducted a statistical evaluation using random sampling within the decision variable space. Fig. \ref{fig6} presents the distribution of overestimation errors for the three benchmark examples under the initial interval configuration. The results demonstrate that the proposed method exhibits consistent superiority across all examples, characterized by a substantially reduced median error and markedly smaller dispersion, in contrast to the Taylor model method, which suffers from large error fluctuations due to the dependency effect. These statistical observations verify that, by effectively suppressing the dependency effect, the proposed algorithm constructs considerably tighter path constraint upper bounds. Consequently, it yields less conservative subproblems, thereby avoiding unnecessary interval refinements triggered by loose bounds. This mechanism substantially reduces the scale of constraints in subsequent iterations, directly explaining the significant improvement in computational efficiency reported in Table \ref{tab1}.

\begin{figure*}[htp] 
	\centering
	\begin{subfigure}[b]{0.328\linewidth} 
		\includegraphics[width=\linewidth]{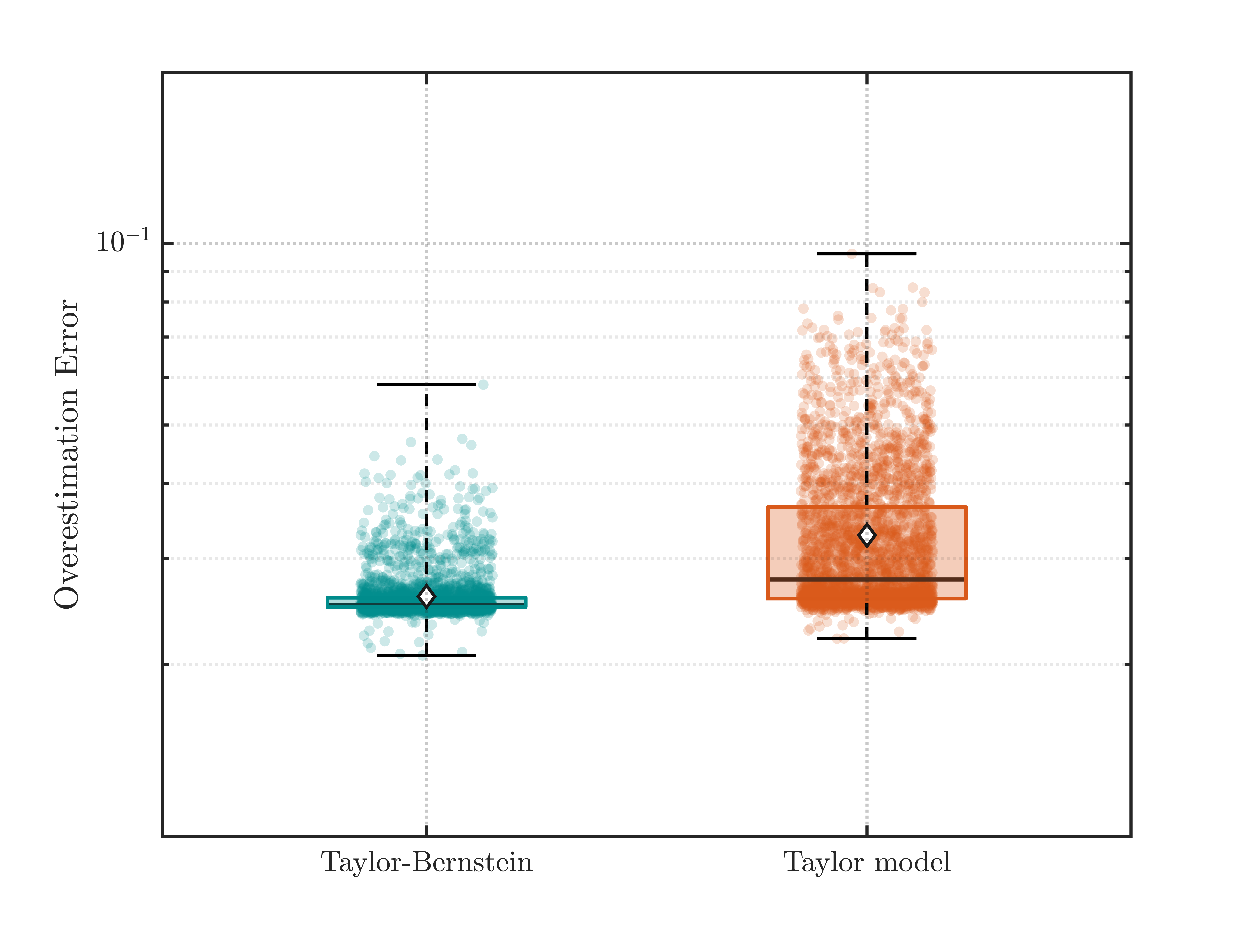}
		\caption{Example 1 }
	\end{subfigure}
	\hfill
	\begin{subfigure}[b]{0.328\linewidth}
		\includegraphics[width=\linewidth]{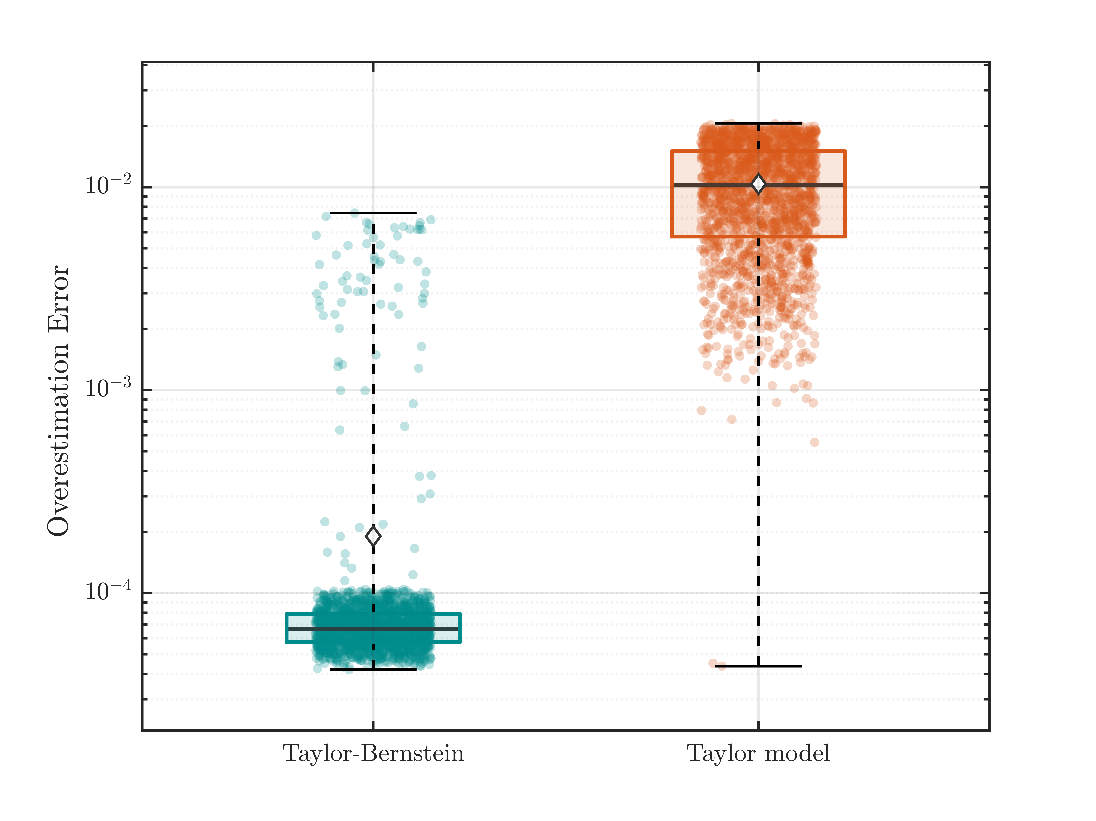}
		\caption{Example 2  }
	\end{subfigure}
	\hfill
	\begin{subfigure}[b]{0.328\linewidth}
		\includegraphics[width=\linewidth]{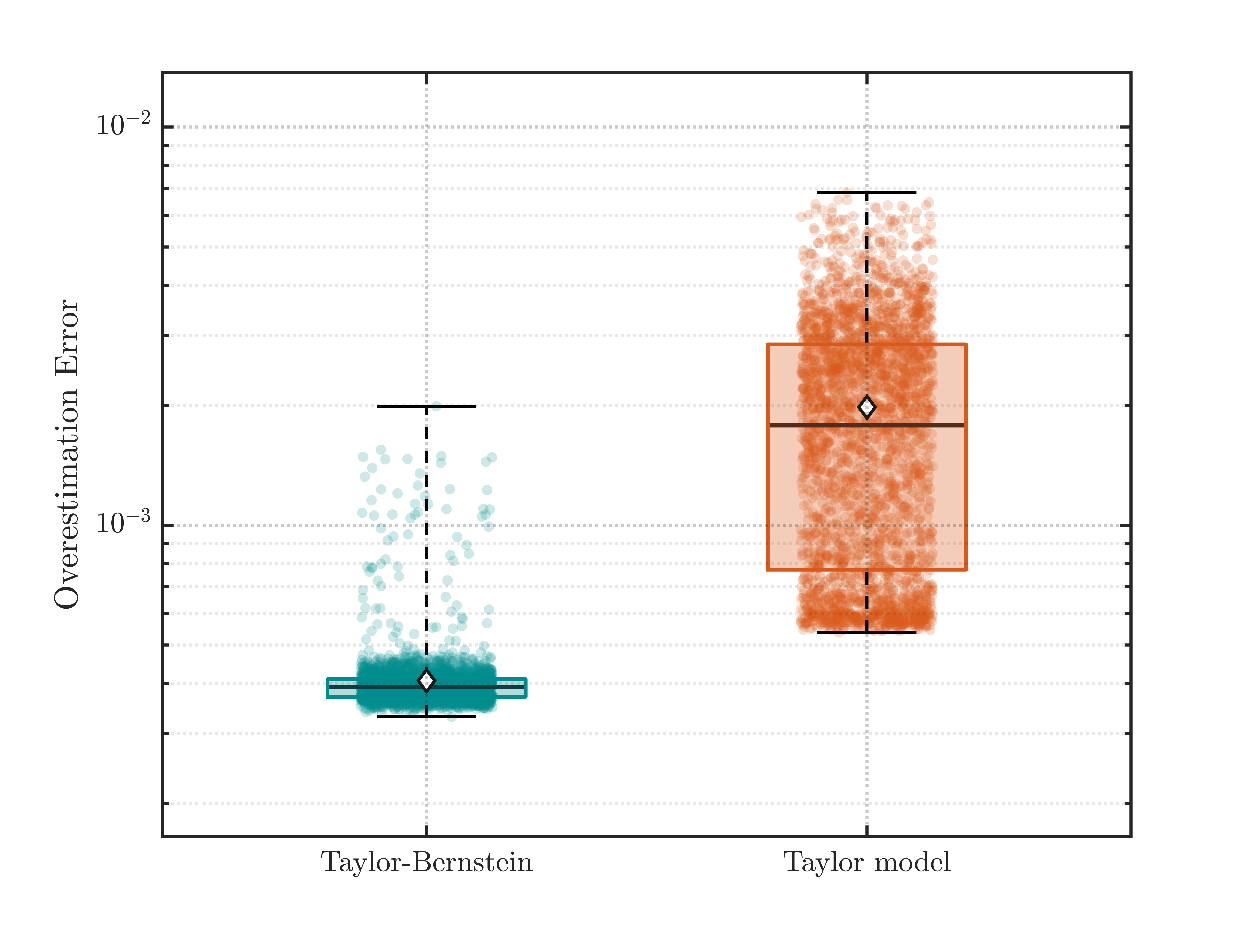}
		\caption{Example 3 }
	\end{subfigure}
	\caption{Statistical comparison of overestimation errors on the initial interval configuration. The central mark indicates the median, the diamond marker represents the mean, and the bottom and top edges of the box represent the 25th and 75th percentiles, respectively.}
	\label{fig6}
\end{figure*}

\section{Conclusion}
This paper presents an inner approximation algorithm for dynamic optimization based on the Taylor-Bernstein form. The core advantage of this method lies in leveraging the convex hull property of Bernstein polynomials to construct strict upper bounds for path constraints that are significantly tighter than those of traditional Taylor models, thereby effectively suppressing the dependency effect in interval arithmetic. Theoretically, the finite convergence of the algorithm is rigorously established by ruling out the infinite recurrence of the algorithmic cases. Numerical simulations on three classic benchmark problems demonstrate that the proposed algorithm significantly reduces the scale of the discretization mesh while guaranteeing strict feasibility. Consequently, it achieves a 30\% to 40\% improvement in computational efficiency compared to existing methods. Future work will explore the application of this framework to non-smooth dynamic optimization and closed-loop dynamic optimization.

\appendix
\section{Numerical Examples}
\label{app1}
\textit{Example 1.} Consider the optimal control problem for the state-constrained Van der Pol oscillator in \cite{26}:
\begin{equation}
	\begin{aligned}
		\operatorname*{min}_{u(t)} \quad & x_3(5)\\
		\mathrm{s.t.} \quad & \dot{x}_1(t)=(1-x^2_2(t))x_1(t)-x_2(t)+u(t),\\
		&\dot{x}_{2}(t)=x_{1}(t),\\
		&\dot{x}_3(t)=x^2_1(t)+x^2_2(t)+u^2(t),\\
		&x(0)=[0,1,0],\\
		&-x_{1}(t)\leqslant0.4, \\
		&-0.3 \leqslant u(t) \leqslant 1,\\
		&0\leqslant t \leqslant 5.
	\end{aligned}
\end{equation}

\textit{Example 2.} The following optimal control problem with time-varying constraints is derived from \cite{27}:
\begin{equation}
	\begin{aligned}
		\operatorname*{min}_{u(t)}\quad&\int_0^1\left\{x^2_1(t)+x^2_2(t)+0.005u^2(t)\right\}dt\\
		\mathrm{s.t.}\quad&\dot{x}_1(t)=x_2(t),\\
		&\dot{x}_2(t)=-x_2(t)+u(t),\\
		&x(0)=[0,-1]^\top,\\
		&x_2(t)+0.5-8(t-0.5)^2\leqslant0,\\
		&-20\leqslant u(t)\leqslant20,\\
		&0\leqslant t\leqslant 1.
	\end{aligned}
\end{equation}

\textit{Example 3.} Consider an obstacle problem consisting of two path constraints \cite{35}:

\begin{equation}
	\begin{aligned}
		\operatorname*{min}_{u(t)}\quad&5x_{1}^{2}(t_{\mathrm{f}})+x_{2}^{2}(t_{\mathrm{f}})\\
		\mathrm{s.t.}\quad&\dot{x}_1=x_2,\\
		&\dot{x}_2=u-0.1(1+2x_1^2)x_1,\\
		&x_{3}=9(x_1-1)^2+\left(\frac{x_2-0.4}{0.3}\right)^2-1,\\
		&\boldsymbol{x}(0)=[1,1,3]^{\mathrm{T}},\\
		&-1\leqslant u\leqslant1,\\
		&-x_2(t)\leqslant0.8,\\
		&-x_{3}(t)\leqslant0,\\
		&0\leqslant t\leqslant2.9.
	\end{aligned}
\end{equation}

\bibliographystyle{elsarticle-num} 
\balance
\bibliography{reference.bib}

\end{document}